\documentclass[11pt]{amsart}
\usepackage{amsmath}
\usepackage[psamsfonts]{amssymb}
\usepackage[all]{xy}
\usepackage{graphicx}
\usepackage[T1]{fontenc}
\usepackage[bitstream-charter]{mathdesign}
\usepackage{hyperref}
\usepackage{verbatim}
\usepackage[english]{babel}
\usepackage{color}
\theoremstyle{definition}
\newtheorem{theorem}{Theorem}[section]
\newtheorem{prop}[theorem]{Proposition}
\newtheorem{lemma}[theorem]{Lemma}
\newtheorem{cor}[theorem]{Corollary}

\theoremstyle{remark}
\newtheorem{dfn}[theorem]{Definition}
\newtheorem{remark}[theorem]{Remark}

\numberwithin{equation}{section}

\def\co{\colon\thinspace}

\def\ep{\epsilon}

\def\R{\mathbb{R}}
\def\Z{\mathbb{Z}}

\def\C{\mathbb{C}}

\title{Local rigidity, symplectic homeomorphisms, and coisotropic submanifolds}
\author{Michael Usher}
\address{Department of Mathematics\\University of Georgia\\Athens, GA 30602}
\email{usher@uga.edu}
\begin{document}
\begin{abstract}
We introduce the notion of a point on a locally closed subset of a symplectic manifold being ``locally rigid'' with respect to that subset, prove that this notion is invariant under symplectic homeomorphisms, and show that coisotropic submanifolds are distinguished among all smooth submanifolds by the property that all of their points are locally rigid.  This yields a simplified proof of the Humili\`ere-Leclercq-Seyfaddini theorem on the $C^0$-rigidity of coisotropic submanifolds.  Connections are also made to the ``rigid locus'' that has previously been used in the study of Chekanov-Hofer pseudometrics on orbits of closed subsets under the Hamiltonian diffeomorphism group.
\end{abstract}
\maketitle

\section{Introduction, definitions and invariance} \label{invtsec}

Let $(M,\omega)$ be a symplectic manifold, and $\mathrm{Symp}(M,\omega)$ the group of symplectic diffeomorphisms of $(M,\omega)$.  The well-known Eliashberg-Gromov theorem asserts that, within the diffeomorphism group of $M$, $\mathrm{Symp}(M,\omega)$ is closed with respect to the $C^0$ (compact-open) topology; this motivates defining a symplectic \emph{homeomorphism} of $M$ as a homeomorphism of $M$ which arises as a $C^0$-limit of symplectic diffeomorphisms. 

It is of interest to understand what properties of $M$ and of its subsets are preserved by symplectic homeomorphisms, and to what extent traditional symplectic notions (whether a submanifold is Lagrangian or more generally coisotropic, for instance) can be characterized in terms of symplectic-homeomorphism-\\invariant criteria.  In this note we define what it means for a point on an appropriate subset $N$ of $M$ to be ``rigid'' or ``locally rigid'' with respect to $N$, and we show (in Proposition \ref{invt}) that these notions are invariant under symplectic homeomorphisms and (in Theorem \ref{coisochar}) that, if $N$ happens to be a smooth submanifold, then $N$ is coisotropic if and only if all of its points are locally rigid.  The Humili\`ere-Leclercq-Seyfaddini theorem \cite{HLS} on the $C^0$-rigidity of coisotropic submanifolds follows as an immediate corollary.  We also show in Proposition \ref{rnchar} that our definition of rigidity is closely related to the ``rigid locus'' considered in \cite{U1}, implying symplectic homeomorphism invariance of certain properties related to the interaction of closed subsets with the Hofer norm and   leading in Corollary \ref{monotone}  to a resolution of a question raised in \cite{U2}.   In the sequel paper \cite{U4}, similar ideas are developed in the context of contact geometry in order to prove a partial contact version of the Humili\`ere-Leclercq-Seyfaddini theorem.

Throughout the paper, ``manifold'' means ``manifold without boundary.''  To set up notation for the definitions of rigidity and local rigidity, for an arbitrary symplectic manifold $(M,\omega)$ let $\mathrm{Ham}(M,\omega)$ denote the group of diffeomorphisms that arise as time-one maps $\phi_{H}^{1}$ of the time-dependent Hamiltonian vector fields of compactly supported smooth functions $H\co [0,1]\times M\to \R$.  For any such $H$, write \[ \|H\|=\int_{0}^{1}\left(\max_M H(t,\cdot)-\min_M H(t,\cdot)\right)dt \] and for $\phi\in \mathrm{Ham}(M,\omega)$ write \[ \|\phi\|=\inf\left\{\|H\||\phi_{H}^{1}=\phi\right\}.\]  (Thus $\|\phi\|$ is the Hofer norm of $\phi$; see \emph{e.g.} \cite{Pbook} for an overview.) % but note that the only nontrivial fact from the literature that we use about the Hofer norm in our proof of Corollary \ref{hls} is Lemma \ref{encap} below.)  
Finally if $U$ and $N$ are subsets of $M$ let \[ e_{M}(U,N)=\inf\left\{\|\phi\|\left|\phi\in \mathrm{Ham}(M,\omega),\,\phi(\bar{U})\cap N=\varnothing\right.\right\}.\]

\begin{dfn}\label{rigdfn} Let $(M,\omega)$ be a symplectic manifold, let $N\subset M$ be a closed subset, and let $p\in N$.  We say that $p$ is \textbf{rigid with respect to $N$} if for every  neighborhood of $U$ of $p$ we have $e_M(U,N)>0$.

If instead $N$ is a locally closed subset\footnote{Recall that this means that each point $q\in N$ has a neighborhood $V$ such that $N\cap V$ is closed as a subset of $V$.}  of $M$ and $p\in N$, we say that $p$ is \textbf{locally rigid with respect to $N$} if there is a neighborhood $W$ of $p$ such that $N\cap W$ is closed as a subset of $W$ and $p$ is rigid with respect to $N\cap W$ (considered as a subset of the symplectic manifold $(W,\omega|_W)$).
\end{dfn}

\begin{remark} All smooth submanifolds of $M$ are locally closed as an immediate consequence of the definition of a submanifold, without making any assumptions about compactness or properness.\end{remark}
\begin{remark} In the definition of local rigidity, the Hamiltonians involved in assessing whether or not $p$ is rigid with respect to $N\cap W$ all have compact support contained in $[0,1]\times W$, and so can be regarded as smooth functions either on $[0,1]\times W$ or (by extension by zero) on $[0,1]\times M$.\end{remark}   

Our applications of rigidity and local rigidity to $C^0$ symplectic geometry are based on the following fact.

\begin{prop}\label{invt} If $(M,\omega)$ is a symplectic manifold, $N$ is a locally closed subset, $p\in N$, and $\psi\co M\to M$ is a symplectic homeomorphism, then $p$ is locally rigid with respect to $N$ if and only if $\psi(p)$ is locally rigid with respect to $\psi(N)$.  Moreover the previous sentence remains true if all appearances of ``locally closed'' and ``locally rigid'' are replaced respectively by ``closed'' and ``rigid.'' 
\end{prop}

\begin{proof}
%Maybe more concise:
Note that it is sufficient to prove the forward implication, since then the reverse implication will follow formally as $\psi^{-1}$ is also a symplectic homeomorphism.
We will show that if $W\subset M$ is open, $N\cap W$ is closed in $W$, and $V$ is an open subset of $W$, then \begin{equation}\label{epres} e_{\psi(W)}(V,\psi(N)\cap \psi(W))\geq e_{W}(\psi^{-1}(V),N\cap W). \end{equation}  The second sentence of the proposition then follows by taking $W=M$, while the first sentence follows by taking $W$ such that $e_W(U,N\cap W)>0$ whenever $U$ is a neighborhood of $p$ in $N\cap W$ and setting $U=\psi^{-1}(V)$ for an arbitrary neighborhood $V$ of $\psi(p)$ in $\psi(N)\cap\psi(W)$.

To prove (\ref{epres}),  Let $\{\psi_m\}_{m=1}^{\infty}$ be a sequence of symplectic diffeomorphisms $C^0$-converging to the symplectic homeomorphism $\psi$, and suppose that $\phi\in \mathrm{Ham}(\psi(W),\omega)$ has $\phi(\bar{V})\cap (\psi(N)\cap \psi(W))=\varnothing$,  and hence \[ (\psi^{-1}\phi\psi)(\psi^{-1}(\bar{V}))\cap (N\cap W)=\varnothing.\] Now because \cite[Theorem 4]{Ar} shows that the homeomorphism group of a manifold is a topological group with respect to the compact-open topology, the sequence $\{\psi_{m}^{-1}\phi\psi_m\}_{m=1}^{\infty}$ $C^0$-converges to $\psi^{-1}\phi\psi$, so for all sufficiently large $m$ it also holds that \begin{equation}\label{psimest} (\psi_{m}^{-1}\phi\psi_m)(\psi^{-1}(\bar{V}))\cap (N\cap W)=\varnothing.\end{equation}  Moreover if we choose $H\co [0,1]\times M\to\R$ having support compactly contained in $[0,1]\times \psi(W)$ such that $\phi_{H}^{1}=\phi$, then the continuous function $K\co [0,1]\times M\to \R$ defined by $K(t,x)=H(t,\psi(x))$ will have support compactly contained in $[0,1]\times W$, and hence for $m$ sufficiently large the smooth functions $K_m\co [0,1]\times M\to \R$ defined by  $K_m(t,x)=H(t,\psi_m(x))$ will also have support compactly contained in $[0,1]\times W$.  

A standard calculation shows that $\phi_{K_m}^{1}=\psi_{m}^{-1}\phi\psi_{m}$.  So by (\ref{psimest}) we have $\|\psi_{m}^{-1}\phi\psi_m\|\geq e_{W}(\psi^{-1}(\bar{V}),N)$.  But $\|\psi_{m}^{-1}\phi\psi_m\|=\|\phi\|$, so we have shown that if $\phi\in \mathrm{Ham}(\psi(W),\omega)$ and $\phi(\bar{V})\cap(\psi(N)\cap\psi(W))$ then $\|\phi\|\geq e_W(\psi^{-1}(V),N)$, which is precisely the desired statement (\ref{epres}). \end{proof}

\section{Local rigidity and coisotropic submanifolds}

The following result gives an interpretation of what it means for a submanifold of a symplectic manifold to be coisotropic in terms that do not reference the tangent spaces to the submanifold (compare \cite[Section 6]{HLS}):

\begin{theorem}\label{coisochar}
Let $(M,\omega)$ be a symplectic manifold and $N$ a smooth submanifold.  Then $N$ is coisotropic if and only if every point $p\in N$ is locally rigid with respect to $N$.
\end{theorem}

As already noted, this quickly implies the following:
\begin{cor} \label{hls} \cite{HLS}  If $(M,\omega)$ is a symplectic manifold, $N$ is a coisotropic submanifold, and $\psi\co M\to M$ is a symplectic homeomorphism with the property that $\psi(N)$ is a smooth submanifold, then $\psi(N)$ is also coisotropic.
\end{cor}

\begin{proof}[Proof of Corollary \ref{hls}, assuming Theorem \ref{coisochar}]
By the forward implication of Theorem \ref{coisochar}, every point of $N$ is locally rigid with respect to $N$, so by Proposition \ref{invt} every point of $\psi(N)$ is locally rigid with respect to $\psi(N)$, which implies that $\psi(N)$ is coisotropic  by the backward implication of Theorem \ref{coisochar}.
\end{proof}

We will prove Theorem \ref{coisochar} momentarily. It seems to the author that this proof of Corollary \ref{hls} is significantly simpler than the original one, though he readily acknowledges that the overall outline of the argument in \cite{HLS} was a considerable influence on the argument given here. After the first version of this paper was finished, Seyfaddini explained to the author another short proof of the special case of Corollary \ref{hls} for closed Lagrangian submanifolds; his argument, like ours, has Lemma \ref{encap} as its main technical ingredient, though it is used in a somewhat different way.  Note that the Eliashberg-Gromov theorem can be seen as a special case of Corollary \ref{hls} because a diffeomorphism of $M$ belongs to $\mathrm{Symp}(M,\omega)$ if and only if its graph is Lagrangian in $(M\times M,(-\omega)\times\omega)$.   
In \cite{HLS} it is also shown that, under the hypotheses of Corollary \ref{hls}, $\psi$ maps the characteristic foliation of $N$ to the characteristic foliation of $\psi(N)$.  We sketch in Remark \ref{charfol} how to deduce the same conclusion using reasoning about local rigidity. 

The backward implication of Theorem \ref{coisochar} is an easy consequence of the following general statement, which we will prove by a variation on an elementary construction from the proofs of \cite[Lemma 4.3]{U1} and \cite[Lemma 2.2]{U2}:

\begin{prop}\label{locrigpres}
Let $(M,\omega)$ be a symplectic manifold and let $N\subset M$ be locally closed, and suppose that $p\in N$ is locally rigid with respect to $N$.  Then for every compactly supported smooth function $H\co M\to \R$ such that $H|_N\equiv 0$ there is $\ep>0$ such that the time-$t$ map $\phi_{H}^{t}$ of the Hamiltonian flow of $H$ obeys $\phi_{H}^{t}(p)\in N$ for all $t\in [0,\ep]$.
\end{prop} 

\begin{proof}  Choose an open set $W$ around $p$ such that $N\cap W$ is closed in $W$ and $e_W(U,N\cap W)>0$ for every  neighborhood $U$ of $p$ in $W$. Let $H\co M\to \R$ be any compactly supported smooth function with $H|_N\equiv 0$.  Choose $\ep_0>0$ such that $\phi_{H}^{t}(p)\in W$ for all $t\in [0,\ep_0]$.  Arguing by contradiction, if the conclusion of the proposition were false we could find a positive $T\leq \ep_0$ such that $\phi_{H}^{T}(p)\notin N$.  Since $\phi_{TH}^{t}=\phi_{H}^{tT}$ we can replace $H$ by $TH$ and $\ep_0$ by $\frac{\ep_0}{T}\geq 1$ and thus reduce to the case that $T=1$. 

After this reduction we have $\phi_{H}^{t}(p)\in W$ for all $t\in [0,1]$ and $\phi_{H}^{1}(p)\notin N$.  
Since $W$ is open in $M$ and $N\cap W$ is closed in $W$ we can then find a neighborhood $V$ of $p$ such that $\phi_{H}^{t}(\bar{V})\subset W$ for all $t\in [0,1]$ and $\phi_{H}^{1}(\bar{V})\cap N=\varnothing.$ Now let $\chi\co M\to \R$ be a smooth function whose support is compact and contained in $W$ and which restricts as $1$ to a neighborhood of $\cup_{t\in [0,1]}\phi_{H}^{t}(\bar{V})$.  The latter property implies that $\phi_{\chi H}^{1}(\bar{V})=\phi_{H}^{1}(\bar{V})$ so, setting $K=\chi H$, we have \begin{equation}\label{vdisp} \mathrm{supp}(K)\subset W\mbox{ and }\phi_{K}^{1}(\bar{V})\cap(N\cap W)=\varnothing.\end{equation}

Now for $m\in \Z_+$ let $\beta_m\co \R\to \R$ be a monotone smooth function with $\beta_m(s)=s$ for $|s|\geq\frac{2}{m}$ and $\beta_m(s)=0$ for $|s|\leq \frac{1}{m}$, and let $K_m=\beta_m\circ K$. So $K_m$, like $K$, has support contained in $W$, while because $K|_{N}=0$ we see that $K_m$ vanishes on a neighborhood of $ N$, and hence $\phi_{K_m}^{1}$ restricts to $N$ as the identity.  So (\ref{vdisp}) implies that \[ \left((\phi_{K_m}^{1})^{-1}\phi_{K}^{1}\right)(\bar{V})\cap (N\cap W)=\varnothing.\]  But $(\phi_{K_m}^{1})^{-1} \phi_{K}^{1}$ is the time-one map of the  Hamiltonian $K-\beta_m\circ K$, which only takes values in the interval $[-\frac{2}{m},\frac{2}{m}]$. Thus $e_W(V,N\cap W)\leq \frac{4}{m}$.  That this holds holds for all $m\in \Z_+$ contradicts the fact that $e_W(V,N\cap W)>0$ according to the first sentence of the proof.  This contradiction proves the proposition.
\end{proof}

\begin{proof}[Proof of the backward implication of Theorem \ref{coisochar}]
Proving the contrapositive, if $N$ is not coisotropic then we may choose $p\in N$ and $v\in T_pN^{\omega}\setminus T_pN$, where $T_{p}N^{\omega}$ is the $\omega$-orthogonal complement to $T_pN$.  Then let $H\co M\to\R$ be a compactly supported smooth function such that $H|_N\equiv 0$ and $(dH)_p(v)\neq 0$, as is possible since $v\notin T_pN$.  The Hamiltonian vector field $X_H$ of $H$ then obeys $\omega_{p}(X_H,v)\neq 0$, which implies that $X_H\notin T_pN$ since $v\in T_pN^{\omega}$ and $T_pN=(T_pN^{\omega})^{\omega}$.  But then we will have $\phi_{H}^{t}(p)\notin N$ for all sufficiently small nonzero $t$, which implies by Proposition \ref{locrigpres} that $p$ is not locally rigid with respect to $N$.
\end{proof}

\begin{proof}[Proof of the forward implication of Theorem \ref{coisochar}]  The main ingredient is the following, which is the only piece of ``hard symplectic topology'' used in either this section or the previous one. The proof in \cite{U1} uses arguments modeled on the proof of a conceptually similar result in \cite{Oh97}; see also \cite{Oh18} for a version of Lemma \ref{encap} that gives a more precise lower bound.

\begin{lemma}\label{encap} (\cite[Corollary 4.10]{U1})
Suppose that $L$ is a compact Lagrangian submanifold of a geometrically bounded symplectic manifold $(M,\omega)$.  Then for every open set $U$ with $U\cap L\neq\varnothing$ we have $e_M(U,L)>0$.
\end{lemma}

In the language of this paper it follows immediately that, under the hypotheses of Lemma \ref{encap}, every point of $L$ is rigid with respect to $L$. If the hypotheses are relaxed this is generally no longer true (see Remark \ref{symplectization}) but we can infer the following:

\begin{cor}\label{lagmain} Let $L$ be any Lagrangian submanifold of any symplectic manifold $(M,\omega)$.  Then every point of $L$ is \textbf{locally} rigid with respect to $L$.
\end{cor}

\begin{proof} Using the Darboux-Weinstein theorem (see \emph{e.g.} \cite[Theorem 5.3.18]{AM} for a proof that avoids compactness assumptions on $L$) together with standard coordinate charts for cotangent bundles, we see that if $(M',\omega')$ is any $2n$-dimensional symplectic manifold, if $L'\subset M'$ is a Lagrangian submanifold, and if $p'\in L'$, there are  neighborhoods $W'$ of $p'$ in $M'$ and $\mathcal{O}'$ of the origin $\vec{0}$ in $\C^n$ together with a symplectomorphism $\Phi\co \mathcal{W'}\to\mathcal{O}'$ (with respect to the standard symplectic structure $\omega_0=\sum_jdx_j\wedge dy_j$ on $\C^n=\{\vec{x}+i\vec{y}\}$) such that $\Phi(L'\cap W')=\R^n\cap \mathcal{O}'$ and $\Phi(p)=\vec{0}$.  

Thus if $L,L'$ are each Lagrangian submanifolds of $2n$-dimensional symplectic manifolds $(M,\omega),(M',\omega')$ respectively and if $p\in L$ and $p'\in L'$, there are neighborhoods $W$ of $L$ and $W'$ of $L'$ and a symplectomorphism $\Psi\co W\to W'$ that sends $L\cap W$ to $L'\cap W'$ and $p$ to $p'$.  Apply this with $L$ equal to our given Lagrangian submanifold, with $L'$ equal to the reader's favorite compact nonempty $n$-dimensional Lagrangian submanifold of any geometrically bounded symplectic manifold, and with $p\in L$ and $p'\in L'$ arbitrary.   By the naturality of the Hofer norm it is immediate that, for any neighborhood $U\subset W$ of $p$, we will have \[ e_W(U,L\cap W)=e_{W'}(\Psi(U),L'\cap W').\]  But since the Hofer norm of an element of $\mathrm{Ham}(W',\omega'|_{W'})$ obviously does not increase when it is instead considered as an element of $\mathrm{Ham}(M',\omega')$, we trivially have $e_{W'}(\Psi(U),L'\cap W')\geq e_{M'}(\Psi(U),L')$.  Since $e_{M'}(\Psi(U),L')>0$ by Lemma \ref{encap}, we have thus shown that $e_{W}(U,L\cap W)>0$ for any neighborhood $U$ of $p$ contained in $W$.  Thus $p$ is locally rigid with respect to $L$.
\end{proof}

To complete the proof of Theorem \ref{coisochar} it remains only to generalize Corollary \ref{lagmain} from Lagrangian submanifolds to coisotropic ones.  The fact that Corollary \ref{lagmain} applies even to Lagrangian submanifolds that are not closed as subsets makes this a simple task, based on the following lemma.

\begin{lemma}\label{lagexists} If $N$ is a coisotropic submanifold of a symplectic manifold $(M,\omega)$ and if $p\in N$ then there is a Lagrangian submanifold $L$ of $M$ such that $p\in L\subset N$. \end{lemma}

\begin{proof} Let $P$ be a (non-closed) submanifold of $N$ with $\dim  P+\dim T_pN^{\omega}=\dim N$ that passes through $p$ and is transverse to the characteristic foliation of $N$.  It is easy to see that $P$ is then a symplectic submanifold of $(M,\omega)$.  Let $\Lambda$ be a Lagrangian submanifold of $(P,\omega|_P)$ such that $p\in \Lambda$, as can be found for instance in a Darboux chart for $P$ around $p$.  Then $\Lambda$ is a $\frac{1}{2}(\dim N-\dim P)$-dimensional isotropic submanifold of $N$ such that every smooth function $H$ with $H|_P=0$ has Hamiltonian vector field $X_H$ which is nowhere tangent to $\Lambda$.  So \cite[Theorem 5.3.30]{AM} applies to give a Lagrangian submanifold $L$ with $p\in \Lambda\subset L\subset N$.
\end{proof} 

Now given a coisotropic submanifold $N$ of $(M,\omega)$, and given $p\in N$, let $L$ be as in Lemma \ref{lagexists}.  Corollary \ref{lagmain} gives a neighborhood $W$ of $p$ with $L\cap W$ closed in $W$ such that $e_{W}(U,L\cap W)>0$ whenever $U\subset W$ is a neighborhood of $p$.  This condition persists if $W$ is replaced by a smaller neighborhood of $p$, so we may as well assume that $W$ is small enough that $N\cap W$ is closed in $W$.  But since $L\subset N$ it trivially holds that $e_W(U,L\cap W)\leq e_W(U,N\cap W)$, so we also have $e_W(U,N\cap W)>0$ for all neighborhoods $U$ of $p$ contained in $W$.  So $p$ is locally rigid with respect to $N$, completing the proof of Theorem \ref{coisochar}.
\end{proof}

\begin{remark} \label{charfol} We have now completed the proof of Corollary \ref{hls}. With a little more effort one can also recover the statement from \cite{HLS} that a symplectic homemorphism $\psi\co M\to M$ that maps a coisotropic submanifold $N$ to a smooth (and hence, by Corollary \ref{hls}, coisotropic) submanifold $\psi(N)$ also maps the characteristic foliation of $N$ to the characteristic foliation of $\psi(N)$; let us briefly sketch this.  Since the statement is local, one can restrict attention to an open subset $U$ of $M$ such that $U\cap N$ is the domain of a coordinate chart in which the leaves of the characteristic foliation are given by $\{x\}\times V_x$ with $x$ varying through an open subset of $\R^{2(n-k)}$ and $V_x$ a neighborhood of the origin in $\R^{k}$.  If $\mathcal{F}$ is a foliation on $N\cap U$ consider the graph \[ \Gamma_{\mathcal{F}}=\left\{(p,q)\in (N\cap U)\times (N\cap U)\left|p,q\mbox{ lie on the same leaf of $\mathcal{F}$}\right.\right\}.\] A distinguishing property of the characteristic foliation $\mathcal{F}_{N\cap U}$  is that the associated graph $\Gamma_{\mathcal{F}_{N\cap U}}$ is a Lagrangian submanifold of $(U\times U,(-\omega)\times \omega)$.

Given our symplectic homeomorphism $\psi$, we can apply our results to the symplectic homeomorphism $\psi\times\psi$ of  $M\times M$.  Now $(\psi\times \psi)(\Gamma_{\mathcal{F}_{N\cap U}})$ is the graph $\Gamma_{\psi_*\mathcal{F}_{N\cap U}}$ of the image $\psi_{*}\mathcal{F}_{N\cap U}$ under $\psi$ of the characteristic foliation of $N\cap U$.  If $\Gamma_{\psi_*\mathcal{F}_{N\cap U}}$ is a smooth submanifold then Corollary \ref{hls} immediately implies that it is Lagrangian, which implies that $\psi_*\mathcal{F}_{N\cap U}$ is equal to the characteristic foliation of $\psi(N\cap U)$ and we are done.  Even if  $\Gamma_{\psi_*\mathcal{F}_{N\cap U}}\subset \psi(U)\times\psi(U)$ is not known to be smooth, Proposition \ref{invt} implies that all of its points are locally rigid with respect to it.  Moreover the projection of $\Gamma_{\psi_*\mathcal{F}_{N\cap U}}$ to either factor is the (smooth) coisotropic submanifold $\psi(N\cap U)$.  If $H\co \psi(U)\to\R$ is any compactly supported smooth function with $H|_{\psi(N\cap U)}\equiv 0$, define $\hat{H}\co \psi(U)\times\psi(U)\to\R$ by $\hat{H}(x,y)=H(y)$.  Then $\hat{H}$ vanishes identically on  $\Gamma_{\psi_*\mathcal{F}_{N\cap U}}$, and so Proposition \ref{locrigpres} shows that the Hamiltonian flow of $\hat{H}$ locally preserves $\Gamma_{\psi_*\mathcal{F}_{N\cap U}}$.  But this is equivalent to saying that the Hamiltonian flow of $H$ locally preserves the leaves of $\psi_*\mathcal{F}_{N\cap U}$.  Since a neighborhood of a point $x$ in a leaf of the characteristic foliation $\mathcal{F}_{\psi(N\cap U)}$ is swept out by integral curves passing through $x$ of Hamiltonian vector fields of smooth functions that vanish along $\psi(N\cap U)$, what we have shown implies that the leaves of $\mathcal{F}_{\psi(N\cap U)}$ are contained in the leaves of $\psi_*\mathcal{F}_{N\cap U}$.  Applying the same reasoning with $\psi$ replaced by $\psi^{-1}$ shows that this containment must be an equality.\end{remark}

\section{Rigidity and the rigid locus}

In \cite{U1}, the \textbf{rigid locus} of a closed subset $N$ of a symplectic manifold $(M,\omega)$ was defined to be \[ R_N=\bigcap_{\phi\in \bar{\Sigma}_N}\phi^{-1}(N) \] where $\Sigma_N\subset \mathrm{Ham}(M,\omega)$ is the subgroup consisting of those Hamiltonian diffeomorphisms $\phi$ with $\phi(N)=N$ and $\bar{\Sigma}_N$ is the closure of $\Sigma_N$ with respect to the topology on $\mathrm{Ham}(M,\omega)$ induced by the Hofer norm $\|\cdot\|$.  \cite{U1} demonstrated that $R_N$ was an effective tool for studying the ``Chekanov-Hofer'' pseudometric $\delta_N$ induced by the Hofer norm on the orbit of $N$ under $\mathrm{Ham}(M,\omega)$; in particular, except in the trivial case that $N$ contains a connected component of $M$, \cite[Lemma 4.2]{U1} shows that $\delta_N$ is nondegenerate if and only if $R_N=N$, and $\delta_N\equiv 0$ if and only if $R_N=\varnothing$.  

We have chosen the perhaps-overused word ``rigid'' in Definition \ref{rigdfn} because of the following new result which provides a less opaque equivalent definition of $R_N$:

\begin{prop}\label{rnchar}
Let $N$ be a closed subset of a symplectic manifold $(M,\omega)$.  Then the rigid locus $R_N$ is the set of all points $p\in N$ such that $p$ is rigid with respect to $N$.
\end{prop}

\begin{proof}
First let us show that if $p\notin R_N$ then $p$ is not rigid with respect to $N$.  That $p\notin R_N$ means that there is $\phi\in\bar{\Sigma}_N$ with $\phi(p)\notin N$.  As $N$ is closed there is then a neighborhood $U$ of $p$ with $\phi(\bar{U})\cap N=\varnothing$.  Since $\phi\in\bar{\Sigma}_N$, for every $\ep>0$ there is $\psi\in \Sigma_N$ such that $\|\psi^{-1}\phi\|<\ep$; the fact that $\psi^{-1}(N)=N$ then yields\[ (\psi^{-1}\phi)(\bar{U})\cap N=\psi^{-1}\left(\phi(\bar{U})\cap N\right)=\varnothing \] and so $e_M(U,N)\leq \|\psi^{-1}\phi\|<\ep$.  Since $\ep$ is arbitrary this proves that $p$ is not rigid with respect to $N$.

For the other inclusion, suppose that $p\in R_N$.  If there is an open subset of $M$ that contains $p$ and is entirely contained in $N$ then an energy-capacity inequality such as \cite[Theorem 1.1]{LM} readily implies that $p$ is rigid with respect to $N$, so assume that every neighborhood of $p$ intersects $M\setminus N$.  

Fix a neighborhood $U$ of $p$. Then in view of the assumption above there is a smooth path $\gamma\co [0,1]\to U$ such that $\gamma(0)=p$ and $\gamma(1)\in M\setminus N$.  Let $\psi\in \mathrm{Ham}(U,\omega|_U)$ be a Hamiltonian diffeomorphism supported inside $U$ such that $\psi(p)=\gamma(1)$ (for instance $\psi$ could be the time-one map of a time-dependent Hamiltonian vector field $X_t$ such that $X_t(\gamma(t))=\gamma'(t)$).  In particular $\psi(p)\notin N$, so since $p\in R_N$ it follows that $\psi\notin \bar{\Sigma}_N$.  So there is $\delta>0$ such that every $\zeta\in \mathrm{Ham}(M,\omega)$ such that $\zeta(N)=\psi(N)$ obeys $\|\zeta\|\geq \delta$ (for otherwise a sequence $\{\zeta_m\}$ with $\zeta_m(N)=\psi(N)$ and $\|\zeta_{m}\|<\frac{1}{m}$  would lead to a sequence $\{\zeta_{m}^{-1}\psi\}$ in $\Sigma_N$ that Hofer-converges to $\psi$).  We claim that $e_M(U,N)\geq \frac{\delta}{2}$.

Indeed, suppose that $\phi\in \mathrm{Ham}(M,\omega)$ has $\phi(\bar{U})\cap N=\varnothing$. Since $\psi$ has support contained in $U$,  $\phi\psi^{-1}\phi^{-1}$ has support contained in $\phi(U)$, and thus acts as the identity on $N$.  Hence $\psi\phi\psi^{-1}\phi^{-1}(N)=\psi(N)$.  So if $\delta$ is as in the previous paragraph, we obtain from standard properties of the Hofer norm that \[ \delta\leq \|\psi\phi\psi^{-1}\phi^{-1}\|\leq \|\psi\phi\psi^{-1}\|+\|\phi^{-1}\|=2\|\phi\|.\] So indeed $e_M(U,N)\geq\frac{\delta}{2}$.  Since $\delta$ depends only on the arbitrary neighborhood $U$ of $p$ this proves that $p$ is rigid with respect to $N$.   
\end{proof}

\begin{cor}
If $\psi\co M\to M$ is a symplectic homeomorphism and if $N\subset M$ is any closed subset then $R_{\psi(N)}=\psi(R_N)$.
\end{cor}
\begin{proof}
This is immediate from Propositions \ref{invt} and \ref{rnchar}.
\end{proof}

In particular it follows from  \cite[Lemma 4.2]{U1} (along with easy arguments in case $N$ contains a connected component of $M$) that the property of the Chekanov-Hofer pseudometric $\delta_N$ being nondegenerate is invariant under symplectic homeomorphisms, and likewise for the property of $\delta_N$ being identically zero.  \cite{U1} and \cite{U2} discuss a variety of cases where one or the other of these properties holds.

\cite[Remark 2.16]{U2} raised the question of whether, as seemed likely based on examples but difficult to prove from the original definition, the rigid locus is monotone with respect to inclusions. Proposition \ref{rnchar} renders this question trivial:

\begin{cor} \label{monotone}
If $(M,\omega)$ is a symplectic manifold and $A,B\subset M$ are closed subsets with $A\subset B$ then $R_A\subset R_B$.
\end{cor}

\begin{proof}
If $U$ is any open set intersecting $A$ one obviously has $e_M(U,A)\leq e_M(U,B)$.  So if $p\in R_A$, \emph{i.e.} (by Proposition \ref{rnchar}) if $p\in A$ with $e_M(U,A)> 0$ for every neighborhood $U$ of $p$, then likewise $e_M(U,B)>0$ for every neighborhood $U$ of $p$, so $p\in  R_B$ by Proposition \ref{rnchar}.
\end{proof}

\begin{remark} \label{symplectization}
If $N\subset M$ is closed, obviously any point $p\in N$ that is rigid with respect to $N$ is also locally rigid with respect to $N$ since one can just take $W=M$ in the definition of local rigidity.  The converse can be quite far from being true.  As a simple example, if $N=\{(x,0)|0\leq x\leq 1\}\subset \R^2$ then it is not hard to see that $R_N=\varnothing$, but 
each point $(x,0)$ with 
$0<x<1$ is locally rigid with respect to $N$.  

It is also possible for \emph{all} points of a closed subset $N\subset M$ to be locally rigid but none to be rigid.  Namely, suppose that $M$ is the symplectization of a contact manifold $Y$ and that $N$ is a compact exact Lagrangian submanifold of $M$;  examples of such $Y$ and $N$ are not easy to construct, but   \cite{Mul} gives an example with $Y$ an exotic contact $\R^5$, and \cite{Mur} gives examples with $Y$ an arbitrary overtwisted contact manifold of dimension at least $5$.  Because $N$ is Lagrangian, all of its points are locally rigid by Corollary \ref{lagmain}.  But the argument in \cite[Section 4]{Che} shows that the Chekanov-Hofer pseudometric $\delta_N$ vanishes identically
and so $R_N=\varnothing$ by \cite[Lemma 4.2]{U1}. Thus Proposition \ref{rnchar} shows that no points of $N$ are rigid.
\end{remark}

\subsection*{Acknowledgements} I thank Sobhan Seyfaddini for his comments on the initial version of the paper. This work was supported by the NSF through the grant DMS-1509213.

%characteristic foliation remark
%Muller, Murphy examples showing difference between local rigidity and rigidity \begin{remark} \label{symplectization}\end{remark}

\end{document}